\documentclass[a4paper,DIV=classic, 11pt]{myart}
\usepackage[normalem]{ulem}

\newcommand{\norm}[1]{\left\Vert#1\right\Vert}
\newcommand{\abs}[1]{\left\vert#1\right\vert}

\renewcommand{\mid}{\,|\,}
\newcommand{\Mid}{\:\big | \:}

\DeclareMathOperator*{\esssup}{ess\,sup}

\usepackage{marginnote}
\usepackage{todonotes}

\begin{document}
\title{BSDEs on finite and infinite horizon with time-delayed generators}
\tnotetext[t]{We thank Michael Kupper for helpful comments and fruitful discussions}

\keyAMSClassification{(2000) 60H20, 28A25.}
\author[a,b,1,u]{Peng Luo}
\author[c,2,s]{Ludovic Tangpi}

 \address[a]{Shandong University, 27 Shanda Nanlu, 250100 Jinan, P.R. China}
 \address[b]{University of Konstanz, Universit\"atsstra\ss e.~10, D-78457 Konstanz, Germany}
 \address[c]{University of Vienna, Faculty of Mathematics, Oskar-Morgenstern-Platz 1, A-1090 Wien, Austria}

 \eMail[1]{peng.luo@uni-konstanz.de}
 \eMail[2]{ludovic.tangpi@univie.ac.at}

\myThanks[u]{Financial support from China Scholarship Council (File No. 201306220101), NSF of China (No.~11221061) and the Project 111 (No. B12023).}
\myThanks[s]{Financial support from Vienna Science and Technology Fund (WWTF) under grant MA 14-008.}

\abstract{
We consider a backward stochastic differential equation with a generator that can be subjected to delay, in the sense that its current value depends on the  weighted past values of the solutions, for instance a distorted recent average.
Existence and uniqueness results are provided in the case of possibly infinite time horizon for equations with, and without reflection.
Furthermore, we show that when the delay vanishes, the solutions of the delayed equations converge to the solution of the equation without delay. 
We argue that these equations are naturally linked to forward backward systems, and we exemplify a situation where this observation allows to derive results for quadratic delayed equations with non-bounded terminal conditions in multi-dimension.
}

\keyWords{Backward stochastic differential equation, delay measure, Weighting function, infinite horizon, barrier, FBSDE.}
\maketitle

\section{Introduction}
In \citet{Del-Imk2010,Del-Imk}, the theory of backward stochastic differential equations (BSDEs) was extended to BSDEs with time delay generators (delay BSDEs).
These are non-Markovian BSDEs in which the generator at each positive time $t$ may depend on the past values of the solutions.
This class of equations turned out to have natural applications in pricing and hedging of insurance contracts, see \citet{Delong}.

The existence result of \citet{Del-Imk}, proved for standard Lipschitz generators and small time horizon, has been refined by \citet{dosReis} who derived additional properties of delay BSDEs such as path regularity and existence of decoupled systems.
Furthermore, existence of delay BSDE constrained above a given continuous barrier has been established by \citet{zhou-ren} in a similar setup.
More recently, \citet{briand-elie} proposed a framework in which quadratic BSDEs with sufficiently small time delay in the value process can be solved.

In addition to the inherent non-Markovian structure of delay BSDEs, the difficulty in studying these equations comes from that the inter-temporal changes of the value and control processes always depend on their entire past, hence making it hard to obtain boundedness of solutions or even BMO-martingale properties of the stochastic integral of the control process.
This suggests that delay BSDEs can actually be solved forward and backward in time and in this regard, share similarities with forward backward stochastic differential equations (FBSDEs), see Section \ref{sec:fbsde} for a more detailed discussion.

The object of the present note is to study delay BSDEs in the case where the past values of the solutions are weighted with respect to some scaling function.
In economic applications, these weighting functions can be viewed as representing the perception of the past of an agent. 
For multidimentional BSDEs with possibly infinite time horizon, we derive existence, uniqueness and stability of delay BSDE in this weighting-function setting.
In particular, we show that when the delay vanishes, the solutions of the delay BSDEs converge to the solution of the BSDE with no delay, hence recovering a result obtained by \citet{briand-elie} for different types of delay.
Moreover, we prove that in our setting existence and uniqueness also hold in the case of reflexion on a c\`adl\`ag barrier.
We observe a link between delay BSDEs and coupled FBSDE and, based on the findings in \citet{qfbsde}, we derive existence of delay quadratic BSDEs in the case where only the value process is subjected to delay.
We refer to \citet{briand-elie} for a similar result, again for a different type of delay and in the one-dimensional case.

In the next section, we specify our probabilistic structure and the form of the equation, then present existence, uniqueness and stability results.
Sections \ref{sec:reflected} and \ref{sec:fbsde} are dedicated to the study of reflected delay BSDEs and the link to FBSDEs, respectively.

\section{BSDEs with time delayed generators}
\label{sec:1}

We work on a filtered probability space $(\Omega, {\cal F}, ({\cal F}_t)_{t \in [0,T]}, P)$ with $T \in (0, \infty]$.
We assume that the filtration is generated by a $d$-dimensional Brownian motion $W$, is complete and right continuous.
Let us also assume that ${\cal F} = {\cal F}_T$.
We endow $\Omega \times [0,T]$ with the predictable $\sigma$-algebra and $\mathbb{R}^k$ with its Borel $\sigma$-algebra.
Unless otherwise stated, all equalities and inequalities between random variables and stochastic processes will be understood in the $P$-a.s. and $P\otimes dt$-a.e. sense, respectively.
For $p \in [1, \infty)$ and $m \in \mathbb{N}$, we denote by ${\cal S}^p(\mathbb{R}^m)$ the space of predictable and continuous processes $X$ valued in $\mathbb{R}^m$ such that $\norm{X}_{{\cal S}^p}^p :=  E[(\sup_{t \in [0,T]}\abs{X_t})^p] < \infty$ and by ${\cal H}^p(\mathbb{R}^{m})$ the space of predictable processes $Z$ valued in $\mathbb{R}^{ m\times d}$ such that $\norm{Z}_{{\cal H}^p}^p := E[(\int_0^T\abs{Z_u}^2\,du)^{p/2}] < \infty$.
For a suitable integrand $Z$, we denote by $Z\cdot W$ the stochastic integral $(\int_0^tZ_u\,dW_u)_{t \in [0,T]}$ of $Z$ with respect to $W$. From \citet{protter01}, $Z\cdot W$ defines a continuous martingale for every $Z \in {\cal H}^p(\mathbb{R}^{m})$.
Processes $(\phi_t)_{t\in[0,T]}$ will always be extended to $[-T,0)$ by setting $\phi_t=0$ for  $t \in [-T,0)$.
We equip $\overline{\mathbb{R}}$ with the $\sigma$-algebra ${\cal B}(\overline{\mathbb{R}})$ consisting of Borel sets of the usual real line with possible addition of the points $-\infty, + \infty$, see \citet{bogachev}.

Let $\xi$ be an ${\cal F}_T$-measurable terminal condition and $g$ an $\mathbb{R}^m$-valued function.
Given two measures $\alpha_1$ and $\alpha_2$ on $[-\infty, \infty]$, and two weighting functions $u,v:[0,T]\to \mathbb{R}$, we study existence of the BSDE
\begin{equation}\label{bsde}
Y_t=\xi+\int_{t}^{T}g(s,\Gamma(s))ds-\int_{t}^{T}Z_sdW_s,\quad t\in[0,T],
\end{equation}
where
\begin{equation}\label{gamma}
\Gamma(s):=\left(\int_{-T}^{0}u(s+r)Y_{s+r}\alpha_1(dr),\int_{-T}^{0}v(s+r)Z_{s+r}\alpha_2(dr)\right).
\end{equation}
\begin{example}
{\bf{1.} }\emph{BSDE with infinite horizon:} If $u = v = 1$ and $\alpha_1 = \alpha_2 = \delta_0$ the Dirac measure at $0$, then Equation \eqref{bsde} reduces to the classical 
BSDE with infinite time horizon and standard Lipschitz generator.\\
{\bf{2.} }\emph{Pricing of insurance contracts:} Let us consider the pricing problem of an insurance contract $\xi$ written on a weather derivative.
It is well known, see for instance \cite{Ank-Imk-Reis} that such contracts can be priced by investing in a highly correlated, but tradable derivative.
In the Merton model, assuming that the latter asset has dynamics
\begin{equation*}
	dS_t = S_t (\mu_t dt + \sigma_t dW_t),
\end{equation*}
then the insurer chooses a number $Z_t$ of shares of $S$ to buy at time $t$ and fixes a cost $c_t$ to be paid by the client.
Hence, he seeks to find the price $V_0$ such that 
\begin{equation*}
	dV_t = c_t\,dt + Z_t\sigma_t(dW_t + \theta_t\,dt)
\end{equation*}
with $\theta_t = \sigma^\prime_t(\sigma_t\sigma_t)^{-1}\mu_t$.
It is natural to demand the cost $c_t$ at time $t$ to depend on the past values of the insurance premium $V_t$, for instance to account for historical weather data.
A possible cost criteria is
\begin{equation*}
	c_t := M_t\int_{-T}^0\cos(\frac{2\pi}{P}(t+s))V_{t+s}\,ds
\end{equation*}
where $P$ accounts for the weather periodicity and $M$ is a scaling parameter.
Thus, the insurance premium satisfies the delay BSDE
\begin{equation*}
	V_t = \xi + \int_t^T\left(\int_{-T}^0M_u\cos(\frac{2\pi}{P}(u+s))V_{u+s}\,ds + Z_u\sigma_u\theta_u\right)\, du - \int_t^TZ_u\sigma_u\,dW_u.
\end{equation*}
\end{example}

\subsection{Existence}
Our existence result for the BSDE \eqref{bsde} is obtained under the following assumptions:
\begin{enumerate}[label=(\textsc{A1}),leftmargin=40pt]
	\item $\alpha_1,\alpha_2$ are two deterministic, finite valued measures supported on $[-T,0]$.\label{a1}
\end{enumerate}
\begin{enumerate}[label=(\textsc{A2}),leftmargin=40pt]
	\item $u,v:[0,T]\rightarrow\mathbb{R}$ are Borel measurable functions such that $u \in L^1(dt)$ and $v \in L^2(dt)$.\label{a2}
\end{enumerate}
\begin{enumerate}[label=(\textsc{A3}),leftmargin=40pt]
	\item $g:\Omega\times[0,T]\times\mathbb{R}^m\times\mathbb{R}^{m\times d}\rightarrow\mathbb{R}^m$ is measurable, such that $\int_0^Tg(s,0,0)\,ds \in L^2(\mathbb{R}^m)$ and
 satisfies the standard Lipschitz condition: there exists a constant $K>0$ such that
\begin{equation*}
|g(t,y,z)-g(t,y',z')|\leq K(|y-y'|+|z-z'|)
\end{equation*}
for every $y,y'\in\mathbb{R}^m$ and $z,z'\in\mathbb{R}^{m\times d}$.\label{a3}
\end{enumerate}
\begin{enumerate}[label=(\textsc{A4}),leftmargin=40pt]
	\item $\xi \in L^2(\mathbb{R}^m)$ and is $\mathcal{F}_T$-measurable. \label{a4}
\end{enumerate}

\begin{theorem}\label{exi_bsde}
	Assume \ref{a1}-\ref{a4}.
	If
 	\begin{equation}\label{con}
 		\begin{cases}
 			K^2\alpha^2_1([-T,0])\norm{u}^2_{L^1(dt)} \leq\frac{1}{25},\\
 			K^2\alpha^2_2([-T,0])\norm{v}^2_{L^2(dt)} \leq\frac{1}{25},
 		\end{cases}
 \end{equation}
then BSDE \eqref{bsde} admits a unique solution $(Y,Z)\in\mathcal{S}^2(\mathbb{R}^m)\times\mathcal{H}^2(\mathbb{R}^{m\times d})$.
\end{theorem}
For the proof we need the following lemma on \emph{a priori} estimates of solutions of \eqref{bsde}.

\begin{lemma}[A priori estimation]\label{estimation} Assume \ref{a1}-\ref{a3}.
	For every $\xi,\bar{\xi}\in L^2(\mathbb{R}^m)$, $(y,z),(\bar{y},\bar{z})\in\mathcal{S}^2(\mathbb{R}^m)\times\mathcal{H}^2(\mathbb{R}^{m\times d})$ and $(Y,Z), (\bar{Y},\bar{Z})\in\mathcal{S}^2(\mathbb{R}^m)\times\mathcal{H}^2(\mathbb{R}^{m\times d})$ satisfying
\begin{equation*}
\begin{cases}
Y_t=\xi+\int_t^Tg(s,\gamma(s))ds-\int_{t}^{T}Z_sdW_s\\
\bar{Y}_t=\bar{\xi}+\int_t^Tg(s,\bar{\gamma}(s))ds-\int_{t}^{T}\bar{Z}_sdW_s, \quad t \in [0,T]
\end{cases}
\end{equation*}
with
\begin{equation*}
\begin{cases}
\gamma(s)=\left(\int_{-T}^{0}u(s+r)y_{s+r}\alpha_1(dr),\int_{-T}^{0}v(s+r)z_{s+r}\alpha_2(dr)\right)\\
\bar{\gamma}(s)=\left(\int_{-T}^{0}u(s+r)\bar{y}_{s+r}\alpha_1(dr),\int_{-T}^{0}v(s+r)\bar{z}_{s+r}\alpha_2(dr)\right).
\end{cases}
\end{equation*}
Then, one has
\begin{align*}
&\|Y-\bar{Y}\|^2_{\mathcal{S}^2(\mathbb{R}^m)}+\|Z-\bar{Z}\|^2_{\mathcal{H}^2(\mathbb{R}^{m\times d})} \leq 20K^2\alpha^2_1([-T,0])\norm{u}^2_{L^1(dt)} \|y-\bar{y}\|^2_{\mathcal{S}^2(\mathbb{R}^m)}\\
&\qquad +10\norm{\xi-\bar{\xi}}^2_{L^2(\mathbb{R}^m)} +20K^2\alpha^2_2([-T,0]) \norm{v}^2_{L^2(dt)}\|z-\bar{z}\|^2_{\mathcal{H}^2(\mathbb{R}^{m\times d})}.
\end{align*}
\end{lemma}
\begin{proof}
Let $(y,z)\in\mathcal{S}^2(\mathbb{R}^m)\times\mathcal{H}^2(\mathbb{R}^{m\times d})$, by assumptions \ref{a1} and \ref{a3},  using $2ab\leq a^2+b^2$ and \cite[Lemma 1.1]{dosReis}, we have
\begin{align*}
&E\left(\int_0^Tg(s,\gamma(s))ds\right)^2 \leq E\left(\int_{0}^{T}\right.|g(s,0,0)|ds+K\int_{0}^T\int_{-T}^0|u(s+r)||y_{s+r}|\alpha_1(dr)ds\\
&\quad  +K\int_{0}^T\left.\int_{-T}^0|v(s+r)||z_{s+r}|\alpha_2(dr)ds\right)^2\\
&\leq 3E\left[\left(\int_{0}^{T}|g(s,0,0)|ds\right)^2+K^2\left(\int_{0}^T\int_{-T}^0|u(s+r)||y_{s+r}|\alpha_1(dr)ds\right)^2\right.\\
&\quad \left.+K^2\left(\int_{0}^T\int_{-T}^0|v(s+r)||z_{s+r}|\alpha_2(dr)ds\right)^2\right]\\
\end{align*}

\begin{align*}
&\leq 3E\left[\left(\int_{0}^{T}|g(s,0,0)|ds\right)^2+K^2\left(\int_{0}^T\alpha_1([s-T,0])|u(s)||y_{s}|ds\right)^2\right.\\
&\quad \left.+K^2\left(\int_{0}^T\alpha_2([s-T,0])|v(s)||z_{s}|ds\right)^2\right]\\
&\leq 3E\left(\int_{0}^{T}|g(s,0,0)|ds\right)^2+3K^2\alpha^2_1([-T,0])\left(\int_{0}^T|u(s)|ds\right)^2E\left[\sup_{0\leq t\leq T}|y_t|^2\right]\\
&\quad+3K^2\alpha^2_2([-T,0])\left(\int_{0}^T|v(s)|^2ds\right)E\left[\int_0^T|z_s|^2ds\right].
\end{align*}
Hence, it holds $\int_0^Tg(s,\gamma(s))\,ds \in L^2$.

Now, for $t\in[0,T]$, we have
\begin{equation}
\label{eq:diff}
Y_t-\bar{Y}_t=\xi-\bar{\xi}+\int_t^Tg(s,\gamma(s))-g(s,\bar{\gamma}(s))ds-\int_t^{T}Z_s-\bar{Z}_sdW_s
\end{equation}
and taking conditional expectation with respect to $\mathcal{F}_t$ yields
\begin{equation*}
Y_t-\bar{Y}_t=E\left[\xi-\bar{\xi}+\int_t^Tg(s,\gamma(s))-g(s,\bar{\gamma}(s))ds\bigg|\mathcal{F}_t\right].
\end{equation*}
By Doob's maximal inequality and $2ab\leq a^2+b^2$, we obtain
\begin{align*}
E\left[\sup_{0\leq t\leq T}|Y_t-\bar{Y}_t|^2\right]&=E\left(\sup_{0\leq t\leq T}\bigg|E\left[\xi-\bar{\xi}+\int_t^Tg(s,\gamma(s))-g(s,\bar{\gamma}(s))ds\bigg|\mathcal{F}_t\right]\bigg|\right)^2\\
&\leq E\left(\sup_{0\leq t\leq T}E\left[|\xi-\bar{\xi}|+\int_0^T|g(s,\gamma(s))-g(s,\bar{\gamma}(s))|ds\bigg|\mathcal{F}_t\right]\right)^2\\
&\leq 8E\left[|\xi-\bar{\xi}|^2+\left(\int_0^T\abs{g(s,\gamma(s))-g(s,\bar{\gamma}(s))}ds\right)^2\right].
\end{align*}
On the other hand, for $t=0$ in \eqref{eq:diff}, bringing $\int_0^{T}Z_s-\bar{Z}_sdW_s$ to the left hand side, taking square and expectation to both sides and $2ab\leq a^2+b^2$, we have
\begin{align*}
E\left[\int_0^T|Z_t-\bar{Z}_t|^2dt\right]&=E\left(\xi-\bar{\xi}+\int_0^Tg(s,\gamma(s))-g(s,\bar{\gamma}(s))ds\right)^2-|Y_0-\bar{Y}_0|^2\\
&\leq E\left(\xi-\bar{\xi}+\int_0^Tg(s,\gamma(s))-g(s,\bar{\gamma}(s))ds\right)^2\\
&\leq 2E\left[|\xi-\bar{\xi}|^2+\left(\int_0^T|g(s,\gamma(s))-g(s,\bar{\gamma}(s))|ds\right)^2\right].
\end{align*}
By assumption \ref{a3}, using \cite[Lemma 1.1]{dosReis} and the inequality $2ab\leq a^2+b^2$, we have
\begin{align*}
&E\left(\int_0^T|g(s,\gamma(s))-g(s,\bar{\gamma}(s))|ds\right)^2
\leq K^2E\left(\int_0^T\int_{-T}^0\right. |u(s+r)||y_{s+r}-\bar{y}_{s+r}|\alpha_1(dr)ds\\
&\quad +\int_0^T\left.\int_{-T}^0|v(s+r)||z_{s+r}-\bar{z}_{s+r}|\alpha_2(dr)ds\right)^2\\
&=K^2E\left(\int_0^T\alpha_1([s-T,0])|u(s)||y_{s}-\bar{y}_s|ds+\int_0^T\alpha_2([s-T,0])|v(s)||z_{s}-\bar{z}_s|ds\right)^2\\
&\leq 2K^2\alpha^2_1([-T,0])\norm{u}^2_{L^1(dt)}\norm{y-\bar{y}}^2_{{\cal S}^2} + 2K^2\alpha^2_2([-T,0]) \norm{v}^2_{L^2(dt)}\norm{z - \bar{z}}^2_{{\cal H}^2}.
\end{align*}
Hence,
\begin{align*}
&\|Y-\bar{Y}\|^2_{\mathcal{S}^2(\mathbb{R}^m)}+\|Z-\bar{Z}\|^2_{\mathcal{H}^2(\mathbb{R}^{m\times d})}\leq 20K^2\alpha^2_1([-T,0]) \norm{u}^2_{L^1(dt)} \|y-\bar{y}\|^2_{\mathcal{S}^2(\mathbb{R}^m)}\\
&\qquad 10E\left[|\xi-\bar{\xi}|^2\right]+20K^2\alpha^2_2([-T,0])\norm{v}^2_{L^2(dt)}\|z-\bar{z}\|^2_{\mathcal{H}^2(\mathbb{R}^{m\times d})}.
\end{align*}
This concludes the proof.
\end{proof}

\begin{proof}[ of Theorem \ref{exi_bsde}]
Let $(y,z)\in\mathcal{S}^2(\mathbb{R}^m)\times\mathcal{H}^2(\mathbb{R}^{m\times d})$ and define the process $\gamma(s):=\left(\int_{-T}^{0}u(s+r)y_{s+r}\alpha_1(dr),\int_{-T}^{0}v(s+r)z_{s+r}\alpha_2(dr)\right)$.
Similar to Lemma \ref{estimation}, it follows from \ref{a1}-\ref{a4} that
\begin{equation*}
E\left(\xi+\int_0^Tg(s,\gamma(s))ds\right)^2<\infty.
\end{equation*}
According to the martingale representation theorem, there exists a unique $Z\in\mathcal{H}^2(\mathbb{R}^{m\times d})$ such that for all $t\in[0,T]$,
\begin{equation*}
E\left[\xi+\int_0^Tg(s,\gamma(s))ds\bigg|\mathcal{F}_t\right]=E\left[\xi+\int_0^Tg(s,\gamma(s))ds\right]+\int_0^tZ_sdW_s.
\end{equation*}
Putting
\begin{equation*}
Y_t :=E\left[\xi+\int_t^Tg(s,\gamma(s))ds\Mid \mathcal{F}_t\right],\quad 0\leq t\leq T,
\end{equation*}
the pair $(Y,Z)$ belongs to $\mathcal{S}^2(\mathbb{R}^m)\times\mathcal{H}^2(\mathbb{R}^{m\times d})$ and satisfies
\begin{equation*}
Y_t=\xi+\int_{t}^Tg(s,\gamma(s))ds-\int_t^TZ_sdW_s,\quad 0\leq t\leq T.
\end{equation*}
Thus we have constructed a mapping $\Phi$ from $\mathcal{S}^2(\mathbb{R}^m)\times\mathcal{H}^2(\mathbb{R}^{m\times d})$ to itself such that $\Phi(y,z)=(Y,Z)$. Let $(y,z),(\bar{y},\bar{z})\in\mathcal{S}^2(\mathbb{R}^m)\times\mathcal{H}^2(\mathbb{R}^{m\times d})$, and $(Y,Z)=\Phi(y,z)$, $(\bar{Y},\bar{Z})=\Phi(\bar{y},\bar{z})$.
By Lemma \ref{estimation}, we have
\begin{align*}
&\|Y-\bar{Y}\|^2_{\mathcal{S}^2(\mathbb{R}^m)}+\|Z-\bar{Z}\|^2_{\mathcal{H}^2(\mathbb{R}^{m\times d})}\leq 10K^2\alpha^2_1([-T,0])\norm{u}^2_{L^1(dt)} \|y-\bar{y}\|^2_{\mathcal{S}^2(\mathbb{R}^m)}\\
&\quad+10K^2\alpha^2_2([-T,0]) \norm{v}^2_{L^2(dt)}\|z-\bar{z}\|^2_{\mathcal{H}^2(\mathbb{R}^{m\times d})}
\end{align*}
so that if condition \eqref{con} is satisfied, $\Phi$ is a contraction mapping which therefore admits a unique fixed point on the Banach space ${\cal S}^2(\mathbb{R}^m)\times {\cal H}^2(\mathbb{R}^{m\times d})$.
This completes the proof.
\end{proof}

\subsection{Stability}
In this subsection, we study stability of the BSDE \eqref{bsde} with respect to the delay measures. 
In particular, in Corollary \ref{thm:cor} below we give conditions under which a sequence of solutions of BSDEs with time delayed generator converges to the solution of a standard BSDE with no delay.
Given two measures $\alpha$ and $\beta$, we write $\alpha \le \beta$ if $\alpha(A) \le \beta(A)$ for every measurable set $A$.
\begin{theorem}
\label{thm:Stability}
Assume \ref{a2}-\ref{a4}.
For $i=1,2$ and $n \in \mathbb{N}$, let $\alpha^{n}_{i},\alpha_{i}$ be measures satisfying \ref{a1}; with $\alpha^{n}_{i}$ satisfying \eqref{con} in Theorem \ref{exi_bsde}
and such that $\alpha^n_{i}([-T,0])$ converges to $\alpha_{i}([-T,0])$. 
If $\alpha^n_{1}\leq \alpha_1$ (or $\alpha_1\le \alpha^n_{1}$) and $\alpha^n_{2}\leq \alpha_2$ (or $\alpha_2\le \alpha^n_{2}$), then $\|Y^n-Y\|_{{\cal S}^2(\mathbb{R}^m)}\rightarrow 0$ and $\|Z^n-Z\|_{{\cal H}^2(\mathbb{R}^{m\times d})}\rightarrow 0$, where $(Y^n, Z^n)$ and $(Y, Z)$ are solutions of the BSDE \eqref{bsde} with delay given by the measures $(\alpha^n_1,\alpha^n_2)$ and $(\alpha_2, \alpha_2)$, respectively.
\end{theorem}

\begin{proof}
From Theorem \ref{exi_bsde}, for every $n$, there exists a unique solution $(Y^{n},Z^{n})$ to the BSDE \eqref{bsde} with delay given by the measures $(\alpha^n_1,\alpha^n_2)$.
Since $\alpha^{n}_{i}$, $i=1,2$ satisfy \eqref{con} in Theorem \ref{exi_bsde}  and $\alpha^n_{i}([-T,0])$ converges to $\alpha_{i}([-T,0])$,
it follows that $\alpha_i$ satisfy \eqref{con} and by Theorem \ref{exi_bsde} there exists a unique solution $(Y,Z)$ to the BSDE with delay given by $(\alpha_1, \alpha_2)$.
Using
\begin{equation*}
Y^{n}_t-Y_t=\int_t^T g(s,\Gamma^n(s))-g(s,\Gamma(s))ds-\int^T_tZ^{n}_s-Z_sdW_s,
\end{equation*}
it follows similar to the proof of Lemma \ref{estimation} that
\begin{equation*}
E\left[\sup_{0\leq t\leq T}|Y^{n}_t-Y_t|^2\right]\leq 4E\left[\left(\int_0^T |g(s,\Gamma^n(s))-g(s,\Gamma(s))|ds\right)^2\right],
\end{equation*}
and
\begin{align*}
E\left[\int^T_0|Z^{n}_t-Z_t|^2\right]&\leq E\left[\left(\int_0^T |g(s,\Gamma^n(s))-g(s,\Gamma(s))|ds\right)^2\right].
\end{align*}
On the other hand, using $2ab\leq a^2+b^2$, we get
\begin{align*}
&E\left[\left(\int_0^T |g(s,\Gamma^n(s))-g(s,\Gamma(s))|ds\right)^2\right]\\
&\leq 2K^2E\left[\left(\int_0^T \bigg|\int_{-T}^0u(s+r)Y^n_{s+r}\alpha^n_1(dr)-\int_{-T}^0u(s+r)Y_{s+r}\alpha_1(dr)\bigg|ds\right)^2\right]\\
&\quad +2K^2E\left[\left(\int_0^T \bigg|\int_{-T}^0v(s+r)Z^n_{s+r}\alpha^n_2(dr)-\int_{-T}^0v(s+r)Z_{s+r}\alpha_2(dr)\bigg|ds\right)^2\right].
\end{align*}
Without loss of generality, we assume $\alpha_1 \le \alpha^n_1$ and $\alpha_2 \le\alpha^n_2$.
Hence $\alpha^n_i-\alpha_i$, $i=1,2$, define positive measures satisfying \ref{a1}.
 Therefore,
\begin{align*}
&E\left[\left(\int_0^T \bigg|\int_{-T}^0u(s+r)Y^n_{s+r}\alpha^n_1(dr)-\int_{-T}^0u(s+r)Y_{s+r}\alpha_1(dr)\bigg|ds\right)^2\right]\\
&\leq 2E\left[\left(\int_0^T \int_{-T}^0\abs{u(s+r)}\abs{Y^n_{s+r}-Y_{s+r}}\alpha^n_1(dr)ds\right)^2\right]\\
 &\quad +2E\left[\left(\int_0^T \int_{-T}^0\abs{u(s+r)}\abs{Y_{s+r}}(\alpha^n_1-\alpha_1)(dr)ds\right)^2\right].
\end{align*}
Using \cite[Lemma 1.1]{dosReis}, we obtain
\begin{align*}
&E\left[\left(\int_0^T \int_{-T}^0\abs{u(s+r)}\abs{Y^n_{s+r}-Y_{s+r}}\alpha^n_1(dr)ds\right)^2\right]+E\left[\left(\int_0^T \int_{-T}^0\abs{u(s+r)}\abs{Y_{s+r}}(\alpha^n_1-\alpha_1)(dr)ds\right)^2\right]\\
&\leq E\left[\left(\int_0^T \alpha^n_1([s-T,0])\abs{u(s)}\abs{Y^n_{s}-Y_{s}}ds\right)^2\right]+E\left[\left(\int_0^T(\alpha^n_1-\alpha_1)([s-T,0])\abs{u(s)}\abs{Y_{s}}ds\right)^2\right]\\
&\leq \left(\alpha^n_1([-T,0])\right)^2\|u\|^2_{L^1(dt)}\|Y^n-Y\|^2_{\mathcal{S}^2(\mathbb{R}^m)}+\left((\alpha^n_1-\alpha_1)([-T,0])\right)^2\|u\|^2_{L^1(dt)}\|Y\|^2_{\mathcal{S}^2(\mathbb{R}^m)}.
\end{align*}
Similarly, for the control processes we have
\begin{align*}
&E\left[\left(\int_0^T \bigg|\int_{-T}^0v(s+r)Z^n_{s+r}\alpha^n_2(dr)-\int_{-T}^0v(s+r)Z_{s+r}\alpha_2(dr)\bigg|ds\right)^2\right]\\
&\leq 2\left(\alpha^n_2([-T,0])\right)^2\|v\|^2_{L^2(dt)}\|Z^n-Z\|^2_{\mathcal{H}^2(\mathbb{R}^{m\times d})}+2\left((\alpha^n_2-\alpha_2)([-T,0])\right)^2\|v\|^2_{L^2(dt)}\|Z\|^2_{\mathcal{H}^2(\mathbb{R}^{m\times d})}.
\end{align*}
Hence
\begin{align*}
&\|Y^n-Y\|^2_{\mathcal{S}^2(\mathbb{R}^m)}+\|Z^n-Z\|^2_{\mathcal{H}^2(\mathbb{R}^{m\times d})} \leq 20K^2\left(\alpha^n_1([-T,0])\right)^2\|u\|^2_{L^1(dt)}\|Y^n-Y\|^2_{\mathcal{S}^2(\mathbb{R}^m)}\\
&\quad+20K^2\left((\alpha^n_1-\alpha_1)([-T,0])\right)^2\|u\|^2_{L^1(dt)}\|Y\|^2_{\mathcal{S}^2(\mathbb{R}^m)}\\
&\quad +20K^2\left(\alpha^n_2([-T,0])\right)^2\|v\|^2_{L^2(dt)}\|Z^n-Z\|^2_{\mathcal{H}^2(\mathbb{R}^{m\times d})}\\
&\quad +20K^2\left((\alpha^n_2-\alpha_2)([-T,0])\right)^2\|v\|^2_{L^2(dt)}\|Z\|^2_{\mathcal{H}^2(\mathbb{R}^{m\times d})}\\
&\leq \frac{4}{5}\|Y^n-Y\|^2_{\mathcal{S}^2(\mathbb{R}^m)}+\frac{4}{5}\|Z^n-Z\|^2_{\mathcal{H}^2(\mathbb{R}^{m\times d})}+ 20K^2\left((\alpha^n_1-\alpha_1)([-T,0])\right)^2\|u\|^2_{L^1(dt)}\|Y\|^2_{\mathcal{S}^2(\mathbb{R}^m)}\\
&\quad +20K^2\left((\alpha^n_2-\alpha_2)([-T,0])\right)^2\|v\|^2_{L^2(dt)}\|Z\|^2_{\mathcal{H}^2(\mathbb{R}^{m\times d})}.
\end{align*}
Therefore, the result follows from the convergence of $\alpha^n_i([-T,0])$, $i=1,2$.
\end{proof}
The following is a direct consequence of the above stability result. 
We denote by $\delta_0$ the Dirac measure at $0$.
\begin{corollary}
\label{thm:cor}
Assume \ref{a2}-\ref{a4}.
For $i=1,2$ and $n \in \mathbb{N}$ let $\alpha^{n}_{i}$ be measures satisfying \ref{a1} and \eqref{con} in Theorem \ref{exi_bsde}
 and such that $\alpha^n_{i}([-T,0])$ converges to $1$. If $\alpha^n_{1}\leq \delta_0$ (or $\delta_0\le\alpha^n_{1}$) and $\alpha^n_{2}\leq \delta_0$ (or $\delta_0\le\alpha^n_{2}$), then $\|Y^n-Y\|_{{\cal S}^2(\mathbb{R}^m)}\rightarrow 0$ and $\|Z^n-Z\|_{{\cal H}^2(\mathbb{R}^{m\times d})}\rightarrow 0$, where $(Y^n, Z^n)$ is the solution of the BSDE \eqref{bsde} with delay given by $(\alpha_1^n, \alpha^n_2)$ and $(Y,Z)$ is the solution of BSDE without delay.
\end{corollary}
We conclude this section with the following counterexample which shows that the condition $\alpha_1 \le \alpha^n_{1}$ (or $\alpha^n_{1}\leq \alpha_1$) and $\alpha_2 \le \alpha^n_{2}$ (or $\alpha^n_{2}\leq \alpha_2$) is needed in the above theorem.
\begin{example}
	Assume that $m=d=1$.
	We denote by $\delta_0$ and $\delta_{-1}$ the Dirac measures at $0$ and $-1$, respectively.
	It is clear that $\delta_0([-1,0]) = \delta_{-1}([-1,0])$.
	Consider the delay BSDEs
	\begin{equation}
	\label{eq:nodelay}
		Y_t = 1 + \int_t^11/5\left(\int_{-1}^0Y_{s+r} + Z_{s+r}\right)\delta_{0}(dr)ds - \int_0^1Z_s\,dW_s
	\end{equation}
	and
	\begin{equation}
	\label{eq:strongdelay}
		\bar{Y}_t = 1 + \int_t^11/5\left(\int_{-1}^0\bar{Y}_{s+r} + \bar{Z}_{s+r}\right)\delta_{-1}(dr)ds - \int_0^1\bar{Z}_s\,dW_s.
	\end{equation}
	Since BSDE \eqref{eq:strongdelay} takes the form $\bar{Y}_t=1 - \int_t^1\bar{Z}_u\,dW_s$, it follows that $\bar{Y}_t = 1$ for all $t\in [0,1]$.
	On the other hand, \eqref{eq:nodelay} is the standard BSDE without delay, its solution can be written as $Y_t= E[H^t_1\mid{\cal F}_t]$, where the deflator $(H^t_s)_{s\ge t}$ at time $t$ is given by $dH^t_s = -\frac{H^t_s}{5}(ds + dW_s)$.
	Thus, $Y_t = \text{exp}(-1/5(1-t))$ and for $t \in [0,1)$, $Y_t < \bar{Y}_t$.
\end{example}

\section{Reflected BSDEs with time-delayed generators}
\label{sec:reflected}
The probabilistic setting and the notation of the previous section carries over to the present one.
In particular, we fix a time horizon $T \in (0, \infty]$ and we assume $m=1$.
For $p\in [1, \infty)$, we further introduce the space ${\cal M}^p(\mathbb{R})$ of adapted c\`adl\`ag processes $X$ valued in $\mathbb{R}$ such that $\norm{X}_{{\cal M}^p}^p :=  E[(\sup_{t \in [0,T]}\abs{X_t})^p] < \infty$
and by ${\cal A}^p(\mathbb{R})$, we denote the subspace of elements of ${\cal M}^p(\mathbb{R})$ which are increasing processes starting at $0$.
Let $(S_t)_{t\in[0,T]}$ be a c\`adl\`ag adapted real-valued process.
In this section, we study existence of solutions $(Y, Z, K)$ of BSDEs reflected on the c\`adl\`ag barrier $S$ and with time-delayed generators.
That is, processes satisfying
\begin{eqnarray}
\label{rbsde}
	&Y_t&=\xi+\int_t^Tg(s,\Gamma(s))ds+K_T-K_t-\int_t^TZ_sdW_s,\quad t\in[0,T]\\
\label{eq:barrier}	&Y&\geq S\\
\label{eq:skor_cond}	&\int_{0}^T&(Y_{t-}-S_{t-})dK_t=0
\end{eqnarray}
with $\Gamma$ defined by \eqref{gamma}.
Consider the condition
\begin{enumerate}[label=(\textsc{A5}),leftmargin=40pt]
	\item $E\left[\sup_{0\leq t\leq T}(S^+_t)^2\right]<\infty$ and $S_T\leq\xi$.\label{a5}
\end{enumerate}
\begin{theorem}
	Assume \ref{a1}-\ref{a5}.
	If
	\begin{equation}\label{con1}
		\begin{cases}
			K^2\alpha^2_1([-T,0])\norm{u}_{L^1(dt)}^2\leq\frac{1}{36},\\
			K^2\alpha^2_2([-T,0])\norm{v}^2_{L^2(dt)}\leq\frac{1}{36},
		\end{cases}
 	\end{equation}
 	then RBSDE \eqref{rbsde} admits a unique solution $(Y,Z,K) \in {\cal M}^2(\mathbb{R})\times {\cal H}^2(\mathbb{R}^d)\times {\cal A}^2(\mathbb{R})$ satisfying
 	\begin{equation*}
		Y_t=\esssup_{\tau\in\mathcal{T}_t}E\left[\int_t^{\tau}g(s,\Gamma(s))ds+S_{\tau}\mathbf{1}_{\{\tau<T\}}+\xi\mathbf{1}_{\{\tau=T\}}\bigg|\mathcal{F}_t\right],
	\end{equation*}
	where $\mathcal{T}$ is the set of all stopping times taking values in $[0,T]$ and $\mathcal{T}_t=\{\tau\in\mathcal{T}:\tau\geq t\}$.
\end{theorem}

\begin{proof}
For any given $(y,z)\in\mathcal{M}^2(\mathbb{R})\times\mathcal{H}^2(\mathbb{R}^d)$, similar to the proof of Lemma \ref{estimation}, we have
\begin{equation*}
E\left(\xi+\int_0^Tg(s,\gamma(s))ds\right)^2<\infty
\end{equation*}
with $\gamma$ defined as in Lemma \ref{estimation}.
Hence, from \cite[Theorem 3.3]{Lep-Xu} for $T<\infty$ and \cite[Theorem 3.1]{ouknine} for $T=\infty$ the reflected BSDE
\begin{equation*}
Y_t=\xi+\int_t^Tg(s,\gamma(s))ds+K_T-K_t-\int_t^TZ_sdW_s
\end{equation*}
with barrier $S$ admits a unique solution $(Y,Z,K)$ such that $(Y,Z)\in\mathcal{B}$, the space of processes $(Y,Z)\in\mathcal{M}^2(\mathbb{R})\times\mathcal{H}^2(\mathbb{R}^d)$ such that $Y\geq S$, and $K \in {\cal A}^2(\mathbb{R})$.
Moreover, $Y$ admits the representation
\begin{equation*}
Y_t=\esssup_{\tau\in\mathcal{T}_t}E\left[\int_t^{\tau}g(s,\gamma(s))ds+S_{\tau}\mathbf{1}_{\{\tau<T\}}+\xi\mathbf{1}_{\{\tau=T\}}\bigg|\mathcal{F}_t\right]\quad t \in [0,T].
\end{equation*}
Hence we can define a mapping $\Phi$ from $\mathcal{B}$ to $\mathcal{B}$ by setting $\Phi(y,z):=(Y,Z)$. Let $(y,z),(\bar{y},\bar{z})\in\mathcal{B}$ and $(Y,Z)=\Phi(y,z)$, $(\bar{Y},\bar{Z})=\Phi(\bar{y},\bar{z})$. From the representation, we deduce
\begin{align*}
&|Y_t-\bar{Y}_t|\\
&\leq \esssup_{\tau\in\mathcal{T}_t}E\left[\int_t^{\tau}|g(s,\gamma(s))-g(s,\bar{\gamma}(s))|ds\bigg|\mathcal{F}_t\right]\leq E\left[\int_0^{T}|g(s,\gamma(s))-g(s,\bar{\gamma}(s))|ds\bigg|\mathcal{F}_t\right].
\end{align*}
Doob's maximal inequality implies that
\begin{align*}
E\left[\sup_{0\leq t\leq T}|Y_t-\bar{Y}_t|^2\right] \leq 4E\left[\left(\int_0^{T}|g(s,\gamma(s))-g(s,\bar{\gamma}(s))|ds\right)^2\right].
\end{align*}
Applying It\^{o}'s formula to $|Y_t-\bar{Y}_t|^2$, we obtain
\begin{align*}
|Y_t-\bar{Y}_t|^2 &+\int_t^T|Z_s-\bar{Z}_s|^2ds = 2\int_t^T(Y_s-\bar{Y}_s)(g(s,\gamma(s))-g(s,\bar{\gamma}(s)))ds\\
&\quad +2\int_t^T(Y_{s-}-\bar{Y}_{s-})d(K_s-\bar{K}_s)-2\int_t^T(Y_s-\bar{Y}_s)(Z_s-\bar{Z}_s)dW_s\\
&=2\int_t^T(Y_s-\bar{Y}_s)(g(s,\gamma(s))-g(s,\bar{\gamma}(s)))ds-2\int_t^T(Y_s-\bar{Y}_s)(Z_s-\bar{Z}_s)dW_s\\
&\quad +2\int_t^T(Y_{s-}-S_{s-})dK_{s}-2\int_t^T(Y_{s-}-S_{s-})d\bar{K}_s-2\int_t^T(\bar{Y}_{s-}-S_{s-})dK_s\\
&\quad +2\int_t^T(\bar{Y}_{s-}-S_{s-})d\bar{K}_s.
\end{align*}
Since $(Y,K)$ and $(\bar{Y},\bar{K})$ satisfy \eqref{eq:barrier} and \eqref{eq:skor_cond}, we have
\begin{align*}
|Y_t-\bar{Y}_t|^2+\int_t^T|Z_s-\bar{Z}_s|^2ds& \leq2\int_t^T(Y_s-\bar{Y}_s)(g(s,\gamma(s))-g(s,\bar{\gamma}(s)))ds\\
&\quad-2\int_t^T(Y_s-\bar{Y}_s)(Z_s-\bar{Z}_s)dW_s.
\end{align*}
Hence
\begin{equation*}
E\left[\int_0^T|Z_s-\bar{Z}_s|^2ds\right]\leq E\left[\sup_{0\leq t\leq T}|Y_t-\bar{Y}_t|^2\right]+E\left[\left(\int_0^{T}|g(s,\gamma(s))-g(s,\bar{\gamma}(s))|ds\right)^2\right].
\end{equation*}
In view of the proof of Lemma \ref{estimation}, we deduce
\begin{align*}
\|Y-\bar{Y}\|^2_{\mathcal{M}^2(\mathbb{R})}&+\|Z-\bar{Z}\|^2_{\mathcal{H}^2(\mathbb{R}^{d})}\leq 9E\left[\left(\int_0^T|g(s,\gamma(s))-g(s,\bar{\gamma}(s))|ds\right)^2\right]\\
&\leq 18K^2\alpha^2_1([-T,0])\norm{u}^2_{L^1(dt)} \|y-\bar{y}\|^2_{\mathcal{M}^2(\mathbb{R})}\\
&\quad+18K^2\alpha^2_2([-T,0])\norm{v}^2_{L^2(dt)}\|z-\bar{z}\|^2_{\mathcal{H}^2(\mathbb{R}^{ d})}.
\end{align*}
By condition \eqref{con1}, $\Phi$ is a contraction mapping and therefore it admits a unique fixed point which combined with the associated process $K$ is the unique solution of the RBSDE \eqref{rbsde}.
\end{proof}

\section{Link to coupled FBSDEs}
\label{sec:fbsde}

In this section, we discuss the connection between BSDEs with time-delayed generators and FBSDEs.
We work in the probabilistic setting and with the notation of Section \ref{sec:1}.

Standard methods to solve BSDEs with quadratic growth in the control variable often rely either on boundedness of the control process, see for instance \cite{Richou2012a} and \cite{Cheridito2014}, or on BMO estimates for the stochastic integral of the control process, see for instance \cite{tevzadze}.
However, as shown in \cite{Del-Imk}, solutions of BSDEs with time-delayed generators do not, in general, satisfy boundedness and BMO properties so that new methods are required to solve quadratic BSDE with time-delayed generators.
Recently, \cite{briand-elie} obtained existence and uniqueness of solution for a quadratic BSDE with delay only in the value process.
We show below that using FBSDE theory, it is possible to generalize their results to multidimension and considering a different kind of delay.
Moreover, our argument allows to solve equations with generators of superquadratic growth.

Let $\alpha_1$ be the uniform measure on $[-T,0]$, $\alpha_2$ the Dirac measure at $0$. Put $u(s)=v(s)=1$, for $s\in[0,T]$.
	We are considering the following BSDE with time delay only in the value process:
\begin{equation}\label{ex1}
Y_t=\xi+\int_t^Tg(s,\int_0^sY_rdr,Z_s)ds-\int_t^TZ_sdW_s,\quad t\in[0,T].
\end{equation}
We denote by ${\cal D}^{1,2}$ the space of all Malliavin differentiable random variables and for $\xi \in {\cal D}^{1,2}$ denote by $D_t\xi$ its Malliavin derivative. We refer to \citet{Nualart2006} for a thorough treatment of the theory of Malliavin calculus, whereas the definition and properties of the BMO-space and norm can be found in \cite{Kazamaki01}.
We make the following assumptions:
\begin{enumerate}[label=(\textsc{B1}),leftmargin=40pt]
	\item $g:[0,T]\times\mathbb{R}^m\times\mathbb{R}^{m\times d}\to\mathbb{R}^m$ is a continuous function such that $g^i(y,z)=g^i(y,z^i)$ and there exists a constant $K>0$ as well as a nondecreasing function $\rho:\mathbb{R}_{+}\to\mathbb{R}_{+}$ such that
\begin{align*}
&|g(s,y,z)-g(s,y',z')|\leq K|y-y'|+\rho(|z|\vee|z'|)|z-z'|,\\
&|g(s,y,z)-g(s,y',z)-g(s,y,z')+g(s,y',z')|\leq K(|y-y'|+|z-z'|	)
\end{align*}
for all $s\in[0,T]$, $y,y'\in\mathbb{R}^m$ and $z,z'\in\mathbb{R}^{m\times d}$.\label{b1}
\end{enumerate}
\begin{enumerate}[label=(\textsc{B2}),leftmargin=40pt]
	\item $\xi$ is $\mathcal{F}_T$-measurable such that $\xi\in\mathcal{D}^{1,2}(\mathbb{R}^m)$ and there exist constants $A_{ij}\geq 0$ such that
\begin{equation*}
|D^j_t\xi^i|\leq A_{ij},\quad i=1,\ldots,m;~j=1,\ldots,d,
\end{equation*}
for all $t\in[0,T]$. \label{b2}
\end{enumerate}
\begin{enumerate}[label=(\textsc{B3}),leftmargin=40pt]
	\item $g:\Omega\times[0,T]\times\mathbb{R}^m\times\mathbb{R}^{m\times d}$ is measurable, $g(s,y,z)=f(s,z)+l(s,y,z)$ where $f$ and $l$ are measurable functions with $f^i(s,z)=f^i(s,z^i)$, $i=1,\ldots,m$ and there exists a constant $K\geq 0$ such that
\begin{align*}
&|f(s,z)-f(s,z')|\leq K(1+|z|+|z'|)|z-z'|,\\
&|l(s,y,z)-l(s,y',z')|\leq K|y-y'|+K(1+|z|^{\epsilon}+|z'|^{\epsilon})|z-z'|,\\
&|f(s,z)|\leq K(1+|z|^2),\\
&|l(s,y,z)|\leq K(1+|z|^{1+\epsilon}),
\end{align*}
for some $0\leq\epsilon<1$ and for all $s\in[0,T]$, $y,y'\in\mathbb{R}^{m}$ and $z,z'\in\mathbb{R}^{m\times d}$.\label{b3}
\end{enumerate}
\begin{enumerate}[label=(\textsc{B4}),leftmargin=40pt]
	\item $\xi$ is $\mathcal{F}_T$-measurable such that there exist a constant $K\geq 0$ such that $|\xi|\leq K$.\label{b4}
\end{enumerate}
\begin{enumerate}[label=(\textsc{B5}),leftmargin=40pt]
	\item $g:\Omega\times[0,T]\times\mathbb{R}\times\mathbb{R}^d\to\mathbb{R}$ is progressively measurable, continuous process for any choice of the spatial variables and for each fixed $(s,\omega)\in[0,T]\times\Omega$, $g(s,\omega,\cdot)$ is continuous. $g$ is increasing in $y$ and for some constant $K\geq 0$ such that
\begin{equation*}
|g(s,y,z)|\leq K(1+|z|),
\end{equation*}
for all $s\in[0,T]$, $y\in\mathbb{R}$ and $z\in\mathbb{R}^d$.\label{b5}
\end{enumerate}
\begin{enumerate}[label=(\textsc{B6}),leftmargin=40pt]
	\item $\xi$ is $\mathcal{F}_T$-measurable such that $\xi\in L^2$.\label{b6}
\end{enumerate}
\begin{enumerate}[label=(\textsc{B7}),leftmargin=40pt]
	\item $g:\Omega\times[0,T]\times\mathbb{R}\times\mathbb{R}^d\to\mathbb{R}$ is progressively measurable, continuous process for any choice of the spatial variables and for each fixed $(s,\omega)\in[0,T]\times\Omega$, $g(s,\omega,\cdot)$ is continuous. $g$ is increasing in $y$ and for some constant $K\geq 0$ such that
\begin{equation*}
|g(s,y,z)|\leq K(1+|z|^2),
\end{equation*}
for all $s\in[0,T]$, $y\in\mathbb{R}$ and $z\in\mathbb{R}^d$.\label{b7}
\end{enumerate}
\begin{proposition}
Assume $T \in (0, \infty)$.	
\begin{enumerate}
	\item If \ref{b1}-\ref{b2} are satisfied, then there exists a constant $C\geq 0$ such that for sufficiently small $T$, BSDE \eqref{ex1} admits a unique solution $(Y, Z)\in {\cal S}^2(\mathbb{R}^m)\times {\cal H}^2(\mathbb{R}^{m\times d})$ such that $\abs{Z} \le C$.
	\item If \ref{b3}-\ref{b4} are satisfied, then there exist constants $C_1,C_2 \geq 0$ such that for sufficiently small $T$, BSDE \eqref{ex1} admits a unique solution $(Y, Z) \in {\cal S}^2(\mathbb{R}^m)\times {\cal H}^2(\mathbb{R}^{m\times d})$ such that $\abs{Y} \le C_1$ and $\|Z\cdot dW\|_{\text{BMO}}\leq C_2$.
	\item If $m =d= 1$ and \ref{b5}-\ref{b6} are satisfied, then BSDE \eqref{ex1} admits at least a solution $(Y, Z)\in \mathcal{S}^2(\mathbb{R})\times\mathcal{H}^2(\mathbb{R}^d)$.
	\item If $m =d= 1$ and \ref{b4} and \ref{b7} are satisfied, then BSDE \eqref{ex1} admits at least a solution $(Y, Z)\in \mathcal{S}^2(\mathbb{R})\times\mathcal{H}^2(\mathbb{R}^d)$ such that $Y$ is bounded and $Z\cdot W$ is a BMO martingale.
\end{enumerate}
\end{proposition}
\begin{proof}
	Define the function $b:\mathbb{R}^m\to \mathbb{R}^m$ by setting for $y\in\mathbb{R}^m$, $b^i(y)=y^i$, $i=1,\ldots,m$.
	For $t\in[0,T]$, put
\begin{equation*}
X_t=\int_0^tb(Y_s)ds.
\end{equation*}
Thus BSDE \eqref{ex1} can be written as the coupled FBSDE
\begin{equation}\label{fbsde}
\begin{cases}
X_t=\int_0^tb(Y_s)ds,\\
Y_t=\xi+\int_t^Tg(s,X_s, Z_s)ds-\int_t^TZ_sdW_s
\end{cases}
\end{equation}	
so that 1. and 2. follow from \cite{qfbsde}, and 3. and 4. from \cite{Ant-Ham}.
\end{proof}
The above theorem provides an explanation why it is not enough to solve a time-delayed BSDE backward in time, one actually needs to consider both the forward and backward parts of the solution due to the delay.


\bibliographystyle{abbrvnat}
\bibliography{reference-tdbsde}

\end{document}